\documentclass[12pt,a4paper,oneside]{amsart}

\usepackage{amsfonts, amsmath, amssymb, amsthm, hyperref}
\usepackage{anysize}

\newtheorem{thm}{Theorem}[section]
\newtheorem*{thm*}{Theorem}
\newtheorem{lem}[thm]{Lemma}

\newtheorem{cor}[thm]{Corollary}

\theoremstyle{definition}
\newtheorem{defn}[thm]{Definition}

\theoremstyle{remark}
\newtheorem{rem}[thm]{Remark}

\numberwithin{equation}{section}

\DeclareMathOperator{\hind}{ind}

 \DeclareMathOperator{\cl}{cl}

\newcommand{\pt}{{\rm pt}}
\newcommand{\inte}{\mathop{\rm int}}

\renewcommand{\epsilon}{\varepsilon}
\renewcommand{\phi}{\varphi}

\newcommand{\ctimes}{}

\begin{document}
\title{Waist of the sphere for maps to manifolds}
\author{R.N.~Karasev}
\thanks{The research of R.N.~Karasev is supported by the Dynasty Foundation, the President's of Russian Federation grant MK-113.2010.1, the Russian Foundation for Basic Research grants 10-01-00096 and 10-01-00139, the Federal Program ``Scientific and scientific-pedagogical staff of innovative Russia'' 2009--2013}
\email{r\_n\_karasev@mail.ru}
\address{Roman Karasev, Dept. of Mathematics, Moscow Institute of Physics and Technology, Institutskiy per. 9, Dolgoprudny, Russia 141700}

\author{A.Yu.~Volovikov}
\thanks{The research of A.Yu.~Volovikov was partially supported by the Russian Foundation for Basic Research}

\email{a\_volov@list.ru}
\address{Alexey Volovikov, Department of Higher Mathematics, Moscow State Institute of Radio-Engineering, Electronics and Automation (Technical University), Pr. Vernadskogo 78, Moscow 117454, Russia}

\subjclass[2000]{28A75,52A38,55R80}
\keywords{measure equipartition, the sphere waist theorem, Borsuk--Ulam theorem}

\begin{abstract}
We generalize the sphere waist theorem of Gromov and the Borsuk--Ulam type measure partition lemma of Gromov--Memarian for maps to manifolds.
\end{abstract}

\maketitle

\section{Introduction}

In~\cite{gr2003,mem2009} the sphere waist theorem was proved for a continuous map from a sphere $S^n$ to the Euclidean space $\mathbb R^m$, showing that the preimage of some point is ``large enough''. Here we generalize it for maps from the sphere to any $m$-dimensional manifold. 

Let the sphere $S^n$ be the standard unit sphere in $\mathbb R^{n+1}$. Denote the standard probabilistic measure on $S^n$ by $\mu$, denote by $U_\varepsilon(X)$ the $\varepsilon$-neighborhood of $X\subseteq S^n$ with respect to the standard metric on $S^n$.

\begin{thm}
\label{sph-waist}
Suppose $h \colon S^n \to M$ is a continuous map from the $n$-sphere to $m$-manifold with $m\le n$. In case $m=n$ let the homology map $h_* \colon H_n(S^n, \mathbb F_2)\to H_n(M, \mathbb F_2)$ be trivial.  Then there exists a point $z\in M$ such that for any $\varepsilon > 0$
$$
\mu U_\epsilon(h^{-1}(z)) \ge \mu U_\epsilon S^{n-m}
$$
Here $S^{n-m}$ is the $(n-m)$-dimensional equatorial subsphere of $S^n$, i.e. $S^{n-m} = S^n\cap \mathbb R^{n-m+1}$.
\end{thm}

In this paper we extend the topological reasoning to the case of maps to manifolds; for the geometrical and analytical part of the proof the reader is referred to~\cite{mem2009}.

\section{The corresponding generalization of the Borsuk--Ulam theorem}

Following~\cite{mem2009}, we are going to prove the corresponding analogue of the Borsuk--Ulam theorem first. In fact, we will prove a more general Borsuk--Ulam type theorem, following mostly~\cite{kar2010}.

Let us give some definitions. Consider a compact topological space $X$ with a probabilistic Borel measure $\mu$. Let $C(X)$ denote the set of continuous functions on $X$.

\begin{defn}
A finite-dimensional linear subspace $L\subset C(X)$ is called \emph{measure separating}, if for any $f\neq g\in L$ the measure of the set 
$$
e(f, g) = \{x\in X : f(x) = g(x)\}
$$
is zero.
\end{defn}

In particular, if $X$ is a compact subset of $\mathbb R^n$ (or $S^n$) such that $X=\cl(\inte X)$, $\mu$ is any absolutely continuous measure, then any finite-dimensional space of analytic functions is measure-separating, because the sets $e(f,g)$ always have dimension $<n$ and therefore measure zero. Then for any collection of $q$ elements of a measure-separating subspace we define a partition of $X$.

\begin{defn}
Suppose $F=\{u_1, \ldots, u_q\} \subset C(X)$ is a family of functions such that $\mu (e(u_i,u_j)) = 0$ for all $i\neq j$. The sets (some of them may be empty)
$$
V_i = \{x\in X : \forall j\neq i\ u_i(x)\ge u_j(x)\}
$$
have a zero measure overlap, so they define a partition $P(F)$ of $X$. In case $u_i$ are linear functions on $\mathbb R^n$ we call $P(F)$ a \emph{generalized Voronoi partition}.
\end{defn}

Note that if we consider the standard sphere $S^n\subseteq \mathbb R^{n+1}$, and homogeneous linear functions $F\subset C(\mathbb R^{n+1}) \subset C(S^n)$, then $P(F)$ is always a partition into convex subsets of $S^n$, or a partition consisting of one set equal to the whole $S^n$. The same is true for (non-homogeneous) linear functions on the Euclidean space $\mathbb R^n$.

We have to generalize the notion of a center function from~\cite{mem2009}.

\begin{defn}
Let $L\subset C(X)$ be a finite-dimensional linear subspace of functions. Suppose that for any subset $F\subset L$ such that all sets $\{V_1,\ldots,V_q\} = P(F)$ have nonempty interiors we can assign \emph{centers} $c(V_1),\ldots, c(V_q)\in X$ to the sets. If this assignment is continuous w.r.t. $F$ and equivariant w.r.t. the permutations of functions in $F$ and permutations of points in the sequence $c_1, \ldots, c_q$, we call $c(\cdot)$ a \emph{$q$-admissible center function} for $L$.
\end{defn}

Now we are ready to state the generalization of~\cite[Theorem~3]{mem2009}.

\begin{thm}
\label{gromov-gen}
Suppose $L$ is a measure-separating subspace of $C(X)$ of dimension $n+1$, $\mu_1,\ldots, \mu_{n-m}$ $(n>m)$ are absolutely continuous (w.r.t. the original measure on $X$) probabilistic measures on $X$. Let $q=p^\alpha$ be a prime power, $c(\cdot)$ be a $q$-admissible center function for $L$, and 
$$
h: X\to M
$$
be a continuous map to an $m$-dimensional manifold. Suppose also that the cohomology map $h^* \colon H^i(M, \mathbb F_p)\to H^i(X, \mathbb F_p)$ is a trivial map for $i > 0$. 

Then there exists a $q$-element subset $F\subset L$ such that for every $i=1,\ldots,n-k$ the partition $P(F)$ partitions the measure $\mu_i$ into $q$ equal parts, and we also have 
$$
h(c(V_1)) = h(c(V_2)) =\dots = h(c(V_q))
$$
for $\{V_1, \ldots, V_q\} = P(F)$.
\end{thm}

\begin{rem}
\label{gromov-gen-rem}
It is clear from the proof in Section~\ref{gromov-gen-proof} that instead of $\mu_1$ we can take any ``charge'' (i.e. a measure that can be negative), the only essential requirement is that $\mu_1(X)\neq 0$. This requirement quarantees that all the partition sets $V_1,\ldots, V_q$ have nonempty interiors.

The other measures $\mu_2,\ldots,\mu_{n-k}$ may be replaced by arbitrary functions of the parts $V_1,\ldots, V_q$ that depend continuously on the partition under the assumption that all the interiors of $V_j$ ($j=1,\ldots, q$) are nonempty. 

When $X=\mathbb R^{n+1}, S^n$ and $L$ is a set of linear functions (this is the case needed in the proof of Theorem~\ref{sph-waist}) the partition sets are convex and there are a lot of suitable functions, for example the Steiner measures (compare~\cite{kar2010}). The existence of many admissible center functions is also obvious in this case.
\end{rem}

\section{Topological facts}

In this section we remind some facts from equivariant topology and prove several Borsuk--Ulam--Bourgin--Yang type results that are needed in the proof of Theorem~\ref{gromov-gen}.

\subsection{Representations and transfer}

Let us start from the following typical problem: let $G$ be a finite group, $Y$ be a $G$-space (i.e. a topological space with a continuous action of $G$) and $V$ be a finite-dimensional linear $G$-representation. For any continuous $G$-equivariant map $f \colon Y\to V$ we may guarantee that $f^{-1}(0)$ is non-empty if the $G$-equivariant Euler class of the vector bundle $Y\times V$ is nonzero. Indeed, the map $f$ can be naturally considered as a $G$-equivariant section of this bundle. This Euler class is the natural image of the ``universal'' Euler class 
$$
e(V)\in H_G^{\dim V}(\pt, \mathbb Z_V) = H^{\dim V}(BG, \mathcal O).
$$
In this formula $\pt$ is a one-point space with trivial $G$-action, $\mathbb Z_V$ is the group $\mathbb Z$ considered to have the $G$-action same as the determinant of its action on $V$, and $\mathcal O$ denotes the corresponding to $\mathbb Z_V$ quotient sheaf over $BG$.

When we consider the Euler class $e(V)$ over $Y$ we assume it to be contained in the cohomology $H_G^{\dim V}(Y, \mathbb Z_V)$.

To avoid the twisted cohomology in our particular problem we are going to use the following consideration. Let $F$ be a subgroup of $G$. Then in the above situation we have two Euler classes 
$$
e_G(V) \in H_G^{\dim V}(Y, \mathbb Z_V),\quad e_F(V) \in H_F^{\dim V}(Y, \mathbb Z_V),
$$
where $\mathbb Z_V$ is simultaneously a $G$-module and an $F$-module (see~\cite{skl2006}). If $F$ acts on $V$ with positive determinant then the last class resides in $H_F^{\dim V}(Y, \mathbb Z)$. 

There exists a natural map $\pi^* \colon H_G^*(Y, \mathbb Z_V) \to H_F^*(Y, \mathbb Z_V)$ that takes $e(V)$ in $G$-equivariant cohomology to $e(V)$ in $F$-equivariant cohomology. There also exist the transfer homomorphism~\cite{bred1997}
$$
\pi_! \colon H_F^*(Y, \mathbb Z_V)\to H_G^*(Y, \mathbb Z_V)
$$
such that the composition $\pi_!\circ \pi^*$ is a multiplication by $|G/F|$. If we tensor-multiply $\mathbb Z_V$ by $\mathbb F_p$, and if $p$ is not a divisor of $|G/F|$ then $\pi^*: H_G^*(Y, \mathbb Z_V\otimes \mathbb F_p)\to H_F^*(Y, \mathbb Z_V\otimes \mathbb F_p)$ is a monomorphism. If $F$ acts on $V$ with positive determinant then the last group is $H_F^*(Y, \mathbb F_p)$ without any twist.

\subsection{The permutation group and its Sylow subgroup}

Now we take $G=\Sigma_q$, the permutation (symmetric) group, where $q=p^k$ is a prime power and $\Sigma_q^{(p)}$ is its $p$-Sylow subgroup (for example if $q=p$ is a prime, $p$-Sylow subgroup of $\Sigma_p$ is a cyclic group of prime order $p$). The group $\Sigma_q$ acts on $\mathbb R^q$ by permutations of coordinates. Since the diagonal $\Delta\subset \mathbb R^q$ is a $\Sigma_q$-invariant subspace, $\Sigma_q$ acts on the quotient $\mathbb R^q/\Delta$, which is isomorphic to $\mathbb R^{q-1}$. We denote this representation 
$$
\alpha_q = \{(x_1,\ldots, x_q)\in \mathbb R^q : x_1 + \dots + x_q = 0\}.
$$

\begin{defn}
Denote the Euler class of $\alpha_q$ reduced modulo $p$ by 
$$
\theta\in H^{q-1}(B\Sigma_q, \mathbb Z_{\alpha_q} \otimes \mathbb F_p).
$$
\end{defn}

Note that the $m$-th power of $\theta$ resides in the cohomology with coefficients $\mathbb Z_{\alpha_q} \otimes \mathbb F_p$ if $m$ is odd and $\mathbb F_p$ is $m$ is even.

\begin{defn}
Denote by $\theta_p\in H^{q-1}(B\Sigma_q^{(p)}, \mathbb F_p)$ the image of $\theta$ under the natural map $\pi^*: H^*(B\Sigma_q, \mathbb Z_{\alpha_q} \otimes \mathbb F_p)\to H^*(B\Sigma_q^{(p)}, \mathbb F_p)$. Note that $\theta_p$ resides in non-twisted cohomology, because $\mathbb Z_{\alpha_q} \otimes \mathbb F_2 = \mathbb F_2$ for $p=2$ and $\Sigma_q^{(p)}$ preserves the orientation for odd $p$.
\end{defn}

It is well-known~\cite{venk1961,mm1982} that the powers of $\theta_p$ are all non-trivial in $H^*(B\Sigma_q^{(p)}, \mathbb F_p)$, it may be shown by passing to the elementary Abelian $p$-torus $(\mathbb Z_p)^k \subseteq \Sigma_q^{(p)}$ (note that $q=p^k$). 

Since $\Sigma_q^{(p)}$ is a Sylow subgroup of $\Sigma_q$ then the index $|\Sigma_q/\Sigma_q^{(p)}|$ is not divisible by $p$ and the power $\theta^m$ is nonzero over a $\Sigma_q$-space $Y$ iff the power $\theta_p^m$ is nonzero over $Y$.

\subsection{Configuration spaces}

Remind the definition of the configuration space:

\begin{defn}
Denote by $K^q(\mathbb R^d)\subset (\mathbb R^d)^{\ctimes q}$ the space of all ordered $q$-tuples of distinct elements in $\mathbb R^d$, i.e. the \emph{configuration space} of $\mathbb R^d$. This space has the natural $\Sigma_q$-action and $\Sigma_q^{(p)}$-action.
\end{defn}

Denote the natural image of $\theta_p$ in the equivariant cohomology of any $\Sigma_q^{(p)}$-space by the same letter $\theta_p$, it does not lead to a confusion in this paper. We give a variant of~\cite[Lemma~6]{kar2009} (the essential idea is from~\cite{vass1988}). 

\begin{lem}
\label{conf-sp-eu}
$$
\theta_p^{d-1} \not=0 \in H_{\Sigma_q^{(p)}}^*(K^q(\mathbb R^d), \mathbb F_p).
$$
\end{lem}

\begin{proof}
This is shown for $\theta^{d-1}$ in~\cite{kar2009}, and this result follows from the above reasoning with transfer.
\end{proof}

We also need the following lemma:

\begin{lem}
\label{zero-set}
Suppose $\theta_p^m$ is nonzero over a $\Sigma_q^{(p)}$-space $Y$.
For a $\Sigma_q^{(p)}$-equivariant map 
$$
f\colon Y\to \alpha_q
$$
denote 
$$
Z_f = f^{-1}(0).
$$
Then $\theta_p^{m-1}$ is nonzero over $Z_f$.
\end{lem}

\begin{proof}
Since the map $Y\setminus Z_f\to \alpha_q$ has no zeroes, the class $\theta_p$ is zero over $Y\setminus Z_f$. If we assume $\theta_p^{m-1}|_{Z_f}=0$ we would obtain (it is important here to use the \v{C}ech cohomology) $\theta_p^m=0$ over the entire $Y$.
\end{proof}

\subsection{Remarks on the notation}

Lemma~\ref{zero-set} guarantees that some $\Sigma_q^{(p)}$-orbit of $Y$ is mapped to one point under any continuous map $Y\to \mathbb R^m$, similar to the standard Borsuk--Ulam theorem. Now we are going to extend this result for maps to manifolds. We need to fix some notation first.

The constant coefficients $\mathbb F_p$, where $\mathbb F_p$ is the field with $p$ elements, are suppressed in the notation of ordinary and equivariant cohomology groups. Any other coefficients are always indicated.

Let $X$ be a $G$-space, $A\subset X$ be an invariant subspace and $\alpha \in H^*_G(X)$.
In what follows by $\alpha|_A\in H^*_G(A)$ we denote the image of $\alpha$ under the homomorphism
induced by the inclusion $\iota_A \colon A\subset X$, and by $\eta|_{X}$ the image of $\eta\in H^*_{G}({\rm pt})$ in
$H^*_{G}({X})$ under the homomorphism of the equivariant cohomology induced by the map $X\to {\rm pt}$.

\subsection{The Haefliger class}

Denote by $T$ a $p$-torus group such that $|T|=q$, so $q=p^\alpha$. Consider an embedding of $T$ in $\Sigma_q$ via regular representation (the embedding is not unique, it depends on the ordering of elements of $T$). Assume that $G$ is a $p$-subgroup of $\Sigma_q$ containing $T$ and $\Sigma_q^{(p)}$ is a
$p$-Sylow subgroup of $\Sigma_q$ containing $G$. Thus we have $T\subseteq G\subseteq \Sigma_q^{(p)}$. In main results of this paper $G=\Sigma_q^{(p)}$.

We saw that the Euler class 
$$
\theta=\theta_{\Sigma_q}:=e_{\Sigma_q}(\alpha_q)\in H^{q-1}(B\Sigma_q;\mathcal O)
$$ 
is mapped to $\theta_p\in H^{q-1}_{\Sigma_q^{(p)}}({\rm pt})$. The inclusions $T\subseteq G\subseteq\Sigma_q^{(p)}$ induce the homomorphisms
$$
H^{q-1}_{\Sigma_q^{(p)}}({\rm pt})\to H^{q-1}_G({\rm pt})\to H^{q-1}_T({\rm pt}),
$$
under which we have $\theta_p\to\theta_G\to\theta_T$.

Following~\cite{haef1961} we are going to define the equivariant diagonal class $\gamma_{M,G}\in H^{m(q-1)}_G(M^{\ctimes q})$ possessing the properties given in the following lemma (we denote by $M^{\ctimes q}$ the $q$-th Cartesian power of $M$):

\begin{lem}\label{haefliger-lem}
Let $M$ be an $m$-dimensional topological manifold. There exists a class $\gamma_{M,G}\in
H^{m(q-1)}_{G}(M^{q})$ such that:

1) For any point $y\in M$ the image of $\gamma_{M,G}$ in
$H^{m(q-1)}_{G}(y^{\ctimes q})=H^{m(q-1)}_{G}({\rm pt})$ coincides with
$\theta_G^m$.

2) The image of $\gamma_{M,G}$ in $H^{m(q-1)}_{G}(M^{q}\setminus \Delta)$
is trivial.

3) For an open submanifold $U\subset M$ the image of $\gamma_{M,G}$
under the homomorphism $H^{m(q-1)}_{G}(M^{\ctimes q})\to H^{m(q-1)}_{G}(U^{\ctimes q})$ coincide with
$\gamma_{U,G}$.

4)  $\gamma_{M,G}$ is the image of $\gamma_{M,p}:=\gamma_{M,\Sigma_q^{(p)}}$,
(the Haefliger class corresponding to the Sylow subgroup
$\Sigma_q^{(p)}$) under the restriction homomorphism $H^{m(q-1)}_{\Sigma_q^{(p)}}(M^{\ctimes q})\to H^{m(q-1)}_{G}(M^{\ctimes q})$.
\end{lem}

\begin{proof}  It is sufficient to consider a connected manifold without boundary (if the boundary is nonempty we construct the class for its double and then restrict it to $M$).

Let us first give an explanation for a closed connected orientable smooth manifold $M$. In this case the equivariant Thom class of the normal bundle to the diagonal $\Delta\subset M^{\ctimes q}$ can be considered as an element of the group $H^{m(q-1)}_G(U_{\varepsilon}(\Delta),\partial U_{\varepsilon}(\Delta))$ where $U_{\varepsilon}(\Delta)$ is a tubular $\varepsilon$-neighborhood of $\Delta$ in $M^{\ctimes q}$ where $\varepsilon$ is small enough. This group is isomorphic to $\mathbb F_p$; and by the excision axiom is also isomorphic to $H^{m(q-1)}_G(M^{\ctimes q},M^{\ctimes q}\setminus \Delta)$. For an open ball $U\subset M$ we can take a generator $\xi_{U,G}\in H^{m(q-1)}_G(U^{\ctimes q},U^{\ctimes q}\setminus \Delta(U))$ which is mapped to $\theta_G^m$ under the homomorphism 
$$
\mathbb F_p=H^{m(q-1)}_G(U^{\ctimes q},U^{\ctimes q}\setminus \Delta(U))\to H^{m(q-1)}_G(U^{\ctimes q})=H^{m(q-1)}_G({\rm pt}).
$$ 
The inclusion $U\subset M$ induces an isomorphism 
$$
\mathbb F_p=H^{m(q-1)}_G(M^{\ctimes q},M^{\ctimes q}\setminus \Delta(M))\to H^{m(q-1)}_G(U^{\ctimes q},U^{\ctimes q}\setminus \Delta(U))
$$ 
and we denote by $\xi_{M,G}$ the generator that is mapped to $\xi_{U,G}$. Finally, denote by $\gamma_{M,G}$ the image of $\xi_{M,G}$ in $H^{m(q-1)}_G(M^{\ctimes q})$. The class $\gamma_{M,G}$ possesses all desired properties as can be easily verified.
 
Now consider the orientable case. Then either $p=2$ and $M$ is an arbitrary manifold (orientable or not), or $p>2$ and the manifold $M$ is orientable. 

Let $E_k$ be any orientable manifold with a free action of $\Sigma_q^{(p)}$. For example, we can take $E_k=K^q(\mathbb R^k)$, the configuration space. It is $(k-2)$-connected and when $k\to\infty$ the space $E_k$ approaches $E\Sigma_q^{(p)}$, the total space of the universal bundle $E\Sigma_q^{(p)}\to B\Sigma_q^{(p)}$.

Fixing the orientations on $M$ and $E_k$, we obtain oriented manifolds $\Delta\times E_k$ and $M^{q}\times E_k$ where ${\Delta}=\Delta(M)$ is the diagonal in $M^{q}$ and we identify $\Delta$ with $M$. Since $G$ is a $p$-group, the diagonal action preserves the orientations of these manifolds. Hence we have oriented manifolds
$$
\Delta\times_{G} E_k:=(\Delta\times E_k)/{G}=\Delta\times(E_k/{G}) \ \text{ and } \ M^{q}\times_{G} E_k:=(M^{q}\times E_k)/{G}
$$ 
of dimensions $m+r$ and $mq+r$ respectively where $r=\dim E_k$.

It follows from the Poincar\'e--Lefschetz duality for manifolds $\Delta\times_{G} E_k$ and 
$M^{q}\times_{G}E_k$ that
$$
\begin{aligned}
\mathbb F_p &=H_{m+r}(\Delta\times_{G} E_k)=H_{m+r}(\Delta\times_{G}
E_k)=\\
&=H^{m(q-1)}((M^{q},M^{q}\setminus \Delta)\times _{G}E_k)=H^{m(q-1)}((M^{q},M^{q}\setminus \Delta)\times
_{G}E_k).
\end{aligned}
$$
Here we consider homology with closed supports (defined via infinite cycles). If the connectivity of the manifold
$E_k$ is large enough, then
$$
H^{m(q-1)}_{G}(M^{q})=H^{m(q-1)}(M^{q}\times_{G} E_k).
$$
Similarly we have
$
H^{m(q-1)}_{G}(M^{q}\setminus \Delta)=H^{m(q-1)}((M^{q}\setminus \Delta)\times
_{G}E_k)
$
and also an isomorphism for pairs
$$
H^{m(q-1)}_{G}(M^{q},M^{q}\setminus \Delta)=H^{m(q-1)}((M^{q},M^{
q}\setminus \Delta)\times _{G}E_k).
$$
For a ball $U\subset M$ the class $\xi_{U,G}$ is already defined and we define $\xi_{M,G}\in H^{m(q-1)}_{G}(M^{q},M^{q}\setminus \Delta)$ as a class which is mapped onto $\xi_{U,G}$. Finally we
define $\gamma_{M,G}$ as the image of $\xi_{M,G}$ in $H^{m(q-1)}_{G}(M^{q})$.

Let us consider now the general (nonorientable) case.

Denote by ${\mathcal H}_{\mathbb Z}$ and ${\mathcal H}$ the orientation sheaves of $M$ with fibers $\mathbb Z$ and
$\mathbb F_p$ respectively. We have ${\mathcal H}={\mathcal H}_{\mathbb Z}\otimes \mathbb F_p$. It is easily seen that
${\mathcal H}_{\mathbb Z}\otimes {\mathcal H}_{\mathbb Z}=\mathbb Z$, which is the constant sheaf. Hence ${\mathcal
H}\otimes{\mathcal H}=\mathbb F_p$, i.e. ${\mathcal H}={\mathcal H}^{-1}$.
Now ${\mathcal H}^{{\widehat \otimes}q}$ is the orientation sheaf of the manifold $M^{q}$. It is enough to
consider the case $q$ is odd (i.e. $p$ is odd), where the orientation of $\mathbb F_p$ makes sense. We have
$$
{\mathcal H}^{{\widehat \otimes}q}\mid_{\Delta}={\mathcal H}^{{ \otimes}q}={\mathcal H}\otimes({\mathcal
H\otimes{\mathcal H}})^{{ \otimes}\frac{q-1}{2}}={\mathcal H}\otimes\mathbb F_p^{\otimes\frac{q-1}{2}}={\mathcal H}
$$
Denote by ${\mathcal H'}$ and $\widetilde{\mathcal H}'$ the orientation sheaves of manifolds $\Delta\times E_k$ and
$M^{q}\times E_k$ respectively. As above we have ${\mathcal H}'=\widetilde{\mathcal H}'|_{\Delta\times E_k}$
and $\widetilde{\mathcal H}'\otimes \widetilde{\mathcal H}'=\mathbb F_p$.

Let us denote the orientation sheaf for the manifold
$
\Delta\times_{G} E_k
$
again by $\mathcal H$ and the orientation sheaf for $M^{q}\times_{G} E_k$ by
$\widetilde{\mathcal H}$.
Let us show that $\widetilde{\mathcal H}|_{\Delta\times_{G} E_k}={\mathcal H}$ and $\widetilde{\mathcal
H}\otimes\widetilde{\mathcal H}=\mathbb F_p$ (the constant sheaf on $M^{q}\times_{G} E_k$).
Consider the action of the group $G$ on the total spaces of sheaves ${\mathcal H}'$, $\widetilde{\mathcal
H}'$ and $\widetilde{\mathcal H}'\otimes\widetilde{\mathcal H}'$. The quotient spaces are the total spaces of sheaves
${\mathcal H}$, $\widetilde{\mathcal H}$ and $\widetilde{\mathcal H}\otimes\widetilde{\mathcal H}$ respectively. Now
the first assertion follows easily. To show that the sheaf $\widetilde{\mathcal H}\otimes\widetilde{\mathcal H}$ on
$M^{q}\times_{G} E_k$ is the constant sheaf $\mathbb F_p$ we consider a nonzero global section $s$
of the constant sheaf $\widetilde{\mathcal H}'\otimes \widetilde{\mathcal H}'=\mathbb F_p$ and an element $g\in
\Sigma_q^{(p)}$. The image of $s$ under $g$ is a section $\alpha s$ where $\alpha=\alpha(g)\in \mathbb F_p$. The order
of the element $g$ is a power of $p$, say $p^r$, so $g^{p^r}$ is the identity of the group $G$. Hence
$\alpha^{p^r}=1$ and from Fermat's Little Theorem we obtain that $\alpha=1$. Therefore any section $s$ is mapped to
itself by all elements of $G$ and thus defines a section of $\widetilde{\mathcal
H}\otimes\widetilde{\mathcal H}$, which is nowhere zero (if $s$ is nonzero); hence $\widetilde{\mathcal
H}\otimes\widetilde{\mathcal H}$ is constant (with fiber $\mathbb F_p$).
It follows from the Poincar\'e--Lefschetz duality for manifolds $\Delta\times_{G} E_k$ and $M^{q}\times
_{G}E_k$ that
$$
\begin{aligned}
\mathbb F_p &=H_{m+r}(\Delta\times_{G} E_k;{\mathcal H})=H_{m+r}(\Delta\times_{G}
E_k;\widetilde{\mathcal H})=\\
&=H^{m(q-1)}((M^{q},M^{q}\setminus \Delta)\times _{G}E_k;\widetilde{\mathcal
H}\otimes\widetilde{\mathcal H})=\\
&=H^{m(q-1)}((M^{q},M^{q}\setminus \Delta)\times
_{G}E_k)=H^{m(q-1)}_{G}(M^{q},M^{q}\setminus \Delta).
\end{aligned}
$$
Here $r=\dim E_k$ and we consider homology with closed supports (defined via infinite cycles) and assume that the connectivity of the manifold $E_k$ is large enough.

Hence there exists a nonzero class $\gamma_{M,G}\in H^{m(q-1)}_{G}(M^{q})$ with
trivial image in $H^{m(q-1)}_{G}(M^{q}\setminus \Delta)$.
Similarly for an open $U\subset M$ we have a class $\gamma_{U,G}\in H^{m(q-1)}_{\Sigma_q^{(p)}}(U^{q};\mathbb
F_p)$; and we can assume that in the equivariant cohomology the restriction of $\gamma_{M,G}$ onto $U^{q}$
coincides with   $\gamma_{U,G}$. Note that $U$ is orientable (i.e. the orientation sheaf is constant) if $U$ is small
enough.
If $U\subset M$ is a ball then $\gamma_{U,G}\in H^{m(q-1)}_{G}(U^{q})=
H^{m(q-1)}_{G}({\rm pt})$ coincides up to a constant nonzero factor with $\theta_G^{m}$
since the cohomology $H^{m(q-1)}_{G}(U^{q},U^{q}\setminus \Delta(U))$
is generated by the equivariant Thom class of the normal bundle of $\Delta(U)$ in $U^{q}$.

Let us prove claim~4 of the theorem. Since $\theta_T^k\neq 0$ in $H^{k(q-1)}_T({\rm pt})$
see~\cite{mm1982,venk1961}, and $\theta_T^k$ is the image of $\theta_G^k$
under the homomorphism $H^{k(q-1)}_{\Sigma_q^{(p)}}({\rm pt})\to H^{k(q-1)}_{G}({\rm
pt})$, we have $\theta_G^k\neq 0$ for any $k$.
Therefore $\gamma_{M,G}\neq 0$, hence
$$
\mathbb F_p=H^{m(q-1)}_{\Sigma_q^{(p)}}(M^{\ctimes q},M^{\ctimes q}\setminus\Delta)\to H^{m(q-1)}_{G}(M^{\ctimes q},M^{\ctimes q}\setminus\Delta)=\mathbb F_p
$$
is an isomorphism and the assertion follows.
\end{proof}

\begin{lem}\label{euler-im} Let $X$ be a compact space, or a $CW$-complex. Let $h\colon X\to M$ be a map such that $h^*\colon H^i(M)\to H^i(X)$ is trivial for any $i>0$.
Then

1) ${h^{\ctimes q}}^*\gamma_{M,G} = \theta_G^m$ in $H^{m(q-1)}_G(X^{\ctimes q})$.

2) $\theta_G^m|_{X^{\ctimes q}\setminus P}=0$ where $P=(h^{\ctimes q})^{-1}\Delta(M)$.
\end{lem}

\begin{proof} We give the proof in the case when both $X$ and $M$ are compact.
Technical details for the proof in general case can be found in \cite{vol1992} where
the case $G=T$ was considered.

We use the Nakaoka lemma~\cite{naka1961}, (see also Lemma~1.1 in~\cite{may1970} or Theorem~2.1 in~\cite{lea1997}) to obtain
\begin{equation}                                                                                                            \label{naka-eq}                                                                                                             H_{G}^*(M^{q}) = H^*(BG, H^*(M)^{\otimes q})=H^*(G; H^*(M)^{\otimes q}).
\end{equation}

Let us decompose                                                                                                            $$                                                                                                                          H^*(BG; H^*(M)^{\otimes q}) = H^*(BG)\oplus \mathcal B
$$                                                                                                                          where $H^*(BG) = H^*(BG, H^0(M)^{\otimes q})$, and $\mathcal B$ is
generated by elements of $H^*(M^{q})$ of positive degree. Let us decompose correspondingly
$\gamma_{M,G} = \gamma_0+\gamma_1$. By Lemma~\ref{haefliger-lem} we obtain
$$
\gamma_0=\theta_G^m\in H^{m(q-1)}(G)=H^{m(q-1)}(BG)=H^{m(q-1)}_G({\rm pt}).
$$

The Cartesian power $h^{q} \colon X^{q}\to M^{q}$ induces a zero map in non-equivariant cohomology
in positive degrees by the assumption. Then $\gamma_1$ is mapped to zero under the map                               $$                                                                                                                          {h^{q}}^* \colon H_{G}^*(M^{q}) \to H_{G}^*(X^{q})                                                                                                                             $$                                                                                                                          because the isomorphism (\ref{naka-eq}) is functorial and we also have for compact $X$                                                    $$                                                                                                                         H_{G}^*(X^{q}) = H^*(BG; H^*(X)^{\otimes q})=H^*(G; H^*(X)^{\otimes q}).
$$                                                                                                                          Now it follows that ${h^{q}}^*\gamma_1=0$ and the first assertion of lemma follows.

Since $\gamma_{M,G}|_{M^{\ctimes q}\setminus \Delta}=0$, we obtain
${h^{q}}^*\gamma_{M,G}|_{X^{\ctimes q}\setminus P}=0$
where $P=(h^{q})^{-1}\Delta$. Hence, $\theta_G^m|_{X^{\ctimes q}\setminus P}=0$.
\end{proof}

\begin{rem} Note that the same definition of Haefliger class works also in the case when
$M$ is a cohomological manifold (see~\cite{bred1997}) over the field $\mathbb F_p$ such that
the tensor square of its orientation sheaf (over $\mathbb F_p$) is constant. It follows that
the results of [15] can be extended to the case of maps of $T$-spaces to nonorientable topological manifolds
and to cohomological manifolds under the above condition. \end{rem}

\subsection{Index of $G$-spaces defined by $\theta_G$}

In this section we consider again a $p$-subgroup $G\subseteq\Sigma_q^{(p)}$ containing $T$.

For a $G$-space $X$ let us introduce its index as follows
$$
\hind_{\theta_G} X = \max \{k : \theta_G^k|_X \neq 0\}
$$
Recall that $\theta_G^k|_X\in H^{k(q-1)}_G(X^{\ctimes q})$ where the coefficient field $\mathbb F_p$ of the cohomology group is suppressed from the notation.

This index possesses usual properties of indicies and we mention some of them needed below.

\begin{lem}
1. Monotonicity under equivariant maps of $G$-spaces: If $X\to Y$ is a $G$-map, then $\hind_{\theta_G} X\le \hind_{\theta_G} Y$.

2. Continuity: Let $A\subset X$ be a closed invariant subspace of a $G$-space $X$, then there exists an invariant open neighborhood $U\supset A$ such that $\hind_{\theta_G}A=\hind_{\theta_p}U$.

3. Subadditivity: If $X=A\cup B$ is a union of invariant subspaces then
$$
\hind_{\theta_G}X\le \hind_{\theta_G}A+\hind_{\theta_G}B+1
$$
provided $A$ and $B$ are both open, or $A$ is closed and $B=X\setminus A$.

4. In the case $q=2$, i.e. for $\Sigma_2=\mathbb Z_2=T=G_2$, the index
$\hind_{\theta_2}$ coincides with Yang's homological index introduced by Yang in \cite{yang1954}.
\end{lem}

\begin{proof} Property 1 follows directly from the definition of the index.

Property~2 is a consequence of the continuity property of cohomology (e.g. the \v{C}ech cohomology).

Property 3. Put $\hind_{\theta_G} A=k$ and $\hind_{\theta_G} A=m$.
Then $\theta_G^{k+1}|A=0$ and $\theta_G^{m+1}|B=0$.
If $A$ and $B$ are open it follows from the properties of multiplication of cohomology classes that
$\theta_G^{k+1}\theta_G^{m+1}|_{A\cup B}=\theta_G^{k+m+2}|X=0$, i.e. $\hind_{\theta_G} X\le k+m+1$.

The second statement of property~3 follows from the first one and the continuity property~2 of the index.

Property 4. Consider a space $X$ with a free involution $\tau\colon X\to X$.
Yang's index of $X$ equals maximal
$k$ such that $k$-th power of the characteristic class of the free involution $\tau$ is nontrivial. By
definition the characteristic class of $\tau$ is the first Stiefel--Whitney class of the linear bundle
associated with the covering $\tau\colon X\to X/\tau$ and it is easy to see that this characteristic
class of $\tau$ coincides with
$\theta_2|_X\in H^1_{\mathbb F_2}({\rm pt};\mathbb Z_2)=H^1(X/\tau;\mathbb F_2)$.
\end{proof}

\begin{thm}\label{ind-additive} Let $\varphi\colon Y\to E$ be an equivariant map of $G$-spaces and $P\subset E$ be an
$G$-invariant closed subspace.
If $\hind_{\theta_G}Y=n>m={\rm ind}_{\theta_G}(E\setminus P)$ then
$\varphi^{-1}(P)\not=\emptyset$ and ${\rm ind}_{\theta_G}\varphi^{-1}(P)\ge n-m$.
\end{thm}
\begin{proof}
From property~1 of the index we have $\hind_{\theta_G}(Y\setminus\varphi^{-1}(P))\le \hind_{\theta_G}(E\setminus P)=m$.
This inequality shows that $\varphi^{-1}(P)$ cannot be empty.

Arguing by contradiction assume that $\hind_{\theta_G}\varphi^{-1} P<n-m-1$. Then from property~3
we obtain
$$
\hind_{\theta_G}Y\le\hind_{\theta_G}\varphi^{-1} P+\hind_{\theta_G}(Y\setminus\varphi^{-1}(P)) +1<n-m-1+m+1=n,
$$
so $\hind_{\theta_G}Y<n$ contradicting with the assumption.
\end{proof}

\begin{defn} Suppose $Y$ is a $G$-space. For a $G$-map $\varphi\colon Y\to (X')^{q}$ we put
$$ C(\varphi)=\varphi^{-1}(\Delta) $$
where
$$ \Delta=\Delta(X')= \{(x',\dots,x')\in (X')^{\ctimes q}\,\,:\,\,x'\in Y'\} $$
is the diagonal in $(X')^{\ctimes q}$.
\end{defn}

\begin{defn}
Let $X$ be a metric space, $Y$ be a $G$-invariant subspace of $X^{\ctimes q}$, and $f: X\to X'$ be a continuous
map to a topological space $X'$. Put
$$
B(f)=\{(x_1,\dots,x_q)\in Y\subset X^{\ctimes q}\,\,:\,\,f(x_1)=\dots=f(x_q)\}.
$$
Obviously, $B(f)=C(\varphi)$ for $\varphi=f^{q}\colon X^{\ctimes q}\to (X')^{\ctimes q}$.
\end{defn}

\begin{rem}
We are mainly interested in the case $Y\subset K^q(X)$, where $K^q(X)$ is a configuration spaced based on a space $X$.
However for $G=T$ the case when $Y\subset X^{\ctimes q}$ is an
invariant subspace such that $T$-action on $Y$ has no
fixed points is also interesting. Note that for a $T$-space $X$ there exists an equivariant embedding $X\to X^{\ctimes q}$ (see \cite{vol1992}). For example, for the space $X$ with an action of the cyclic group $\mathbb Z_p$ with
generator $\tau$ this is a map
$x\to (x,\tau x, \dots,\tau^{p-1} x)\in X^p$. In particular for the space $X$ with involution $\tau$
we have the equivariant embedding $x\to (x,\tau x)\in X^2$.
\end{rem}

Since (from the equivariant Thom isomorphism and annihilation of the Euler class)
$$
\hind_{\theta_G}\left((\mathbb R^m)^{q}\setminus\Delta\right) = \hind_{\theta_G} \left( (\alpha_q)^{
m}\setminus\{0\}\right) = m-1,
$$
we obtain:

\begin{cor}\label{bu-euclid}
Let $Y$ be a $G$-space and $\varphi\colon Y\to (\mathbb R^m)^{\ctimes q}$ be a $G$-map.
If $\hind_{\theta_G}Y=n\ge m$ then $C(\varphi)\not=\emptyset$ and
$\hind_\theta C(\varphi)\ge n-m$.

In particular, if $Y\subset X^{\ctimes q}$ is an invariant subspace and
$f\colon X\to\mathbb R^m$   then
$B(f)\not=\emptyset$ and $\hind_\theta B(f)\ge n-m$.
\end{cor}

\begin{rem} The classical Bourgin--Yang
theorem for $\mathbb Z_2$-spaces and maps to Euclidean spaces follows easily from the above result.
\end{rem}

Let $Y$ be a $G$-space and $\varphi\colon Y\to X^{\ctimes q}$ a $G$-map. Let $h\colon X\to M$ be a map.
Then $\psi:=h^{\ctimes q}\circ \varphi\colon Y\to M^{\ctimes q}$ is an equivariant map and
$$
C(\psi)=\varphi^{-1}\circ(h^{q})^{-1}(\Delta(M)).
$$

\begin{thm}\label{bu-manifolds}
Let $Y$ be a $G$-space, $\varphi\colon Y\to X^{\ctimes q}$ a $G$-map and $h\colon X\to M$
be a continuous map to an m-dimensional topological manifold of a space $X$ which
is compact or a $CW$.

Assume that $h^*\colon H^{i}(M)\to H^{i}(X)$ is trivial in dimensions $i>0$.
If $\hind_{\theta_G}Y=n\ge m$ then
$
\hind_{\theta_G}C(\psi)\ge n-m.
$
In particular, $C(\psi)\not=\emptyset$.
\end{thm}

\begin{proof}
Now we can use theorem~\ref{ind-additive}, however it is easier to repeat the argument. From lemma~\ref{euler-im} and monotonicity property of the index we obtain $\hind_{\theta_G}(Y\setminus C(\psi))\le m-1$. Thus $C(\psi)$ cannot be empty.

From Property~3 of the index we obtain
$$
\hind_{\theta_G}Y\le \hind_{\theta_G}C(\psi)+\hind_{\theta_G}(Y\setminus C(\psi))+1\le\hind_{\theta_G}C(\psi)+m.
$$
So, if $\hind_{\theta_G}C(\varphi)< n-m$ then $\hind_{\theta_G}Y<n$, contradicting with our assumptions.
\end{proof}

\section{Proof of Theorem~\ref{gromov-gen}}
\label{gromov-gen-proof}

We apply the above results for the $p$-Sylow group $G=\Sigma_q^{(p)}$.

Using the index notation we have $\hind_{\theta_p}K^q(\mathbb R^d)=d-1$, and Lemma~\ref{zero-set} can be stated
as follows: $\hind_{\theta_p} Z_f \ge \hind_{\theta_p} X - 1$.

For $y\in K^q(L)=K^q(\mathbb R^{n+1})$ we have a partition $(V_1(y),\dots,V_q(y))$ of $X$ and a $\Sigma_q^{(p)}$-map $\varphi_1:K^q(\mathbb R^{n+1})\to \mathbb R^q$ defined as $\varphi_1(y)=(\mu_1(V_1(y)),\dots,\mu_1(V_q(y)))\in \mathbb R^q$.

Put $Y_1=C(\varphi_1)$ and from corollary~\ref{bu-euclid} we obtain $\hind_{\theta_p}Y_1\ge n-1$. Note also that for any $y\in Y_1$ the partition $V_1(y), \ldots, V_q(y)$ consists of sets with nonempty interiors, thus justifying the Remark~\ref{gromov-gen-rem}.

Define
$\varphi_2\colon Y_1\to (\mathbb R^{n-m-1})^{\ctimes q}$ as $\varphi_2(y)=(h(V_1(y)),\dots,h(V_q(y)))$ where
$h=(\mu_{2},\dots,\mu_{n-m})$. Applying
corollary~\ref{bu-euclid} again we see that $\hind_{\theta_p}C(\varphi_2)\ge m$ and we put $Y_2=C(\varphi_2)$.

Finally we define equivariant map $\varphi_3\colon Y_2\to X^{\ctimes q}$ as
$\varphi_3(y)=(c(V_1(y)),\dots,c(V_q(y)))$. Applying theorem~\ref{bu-manifolds} to maps $\varphi_3$ and $f:X\to M$ we
finish the proof.

\begin{rem}
Alternatively we can apply theorem~\ref{bu-manifolds} to maps
$$
\psi:Y_1\to X_1^{\ctimes q}=(X\times \mathbb R^{n-m-1})^{\ctimes q} \ \text{ and } \ f_1\colon X_1\to M_1=M\times \mathbb R^{n-m-1}
$$
defined as
$
\psi(y)=(r(V_1(y)),\dots,r(V_q(y))\ \text{ where }\ r(V_i(y))=(c(V_i(y)),h(V_i(y)))
$ and $f_1=(f, {\rm id}_{\mathbb R^{n-m-1}})$
\end{rem}

\section{Proof of Theorem~\ref{sph-waist}}

First note that for $n=m$ the theorem follows from the ordinary Borsuk--Ulam theorem for maps to manifolds (see~\cite{cf1964} of \cite{vol1992} for example). In this case some two antipodal points $x, -x\in S^n$ are mapped to a single point in $M$, and the set $\{x, -x\}$ is itself a standard $0$-sphere. 

In case $n>m$ the proof follows literally the proof in~\cite{mem2009}. The only thing we have to check is whether Theorem~4 of~\cite{mem2009} can be generalized for maps to $M$, that is we have to prove the following: 

\begin{lem}
\label{pancake}
Suppose $h :S^n \to M$ is a continuous map satisfying assumptions of Theorem~\ref{sph-waist}. Then for any $q=2^l$ the sphere can be partitioned into $q$ convex parts $V_1^q, V_2^q,\ldots, V_q^q$ so that

1) the measures $\mu V_1^q, \mu V_2^q,\ldots, \mu V_q^q$ are equal;

2) the mass centers $c(V_1^q), c(V_2^q),\ldots, c(V_q^q)$ are mapped by $h$ to the same point;

3) for any $\varepsilon>0$ there exists $N$ such that for any $q=2^l>N$ and any $1\le j\le q$ the set $V_j^q$ is $\varepsilon$-close to some $m$-dimensional subsphere of $S^n$ (i.e. intersection $S^n\cap V$ with an $(m+1)$-dimensional linear subspace $V\subset\mathbb R^{n+1}$).
\end{lem}

\begin{proof}
Following~\cite{mem2009} we reproduce the proof of Theorem~\ref{gromov-gen} in a modified form. Consider some linear space $L$ of homogeneous linear functions on $\mathbb R^{n+1}\supset S^n$. The corresponding partitions will be partitions into convex sets.

Let us restrict the symmetry group to $\Sigma_q^{(2)}$ and pass from the configuration space $K^q(L)$ to a certain $\Sigma_q^{(2)}$-invariant subspace $Q^q(L)$, this space was defined explicitly in~\cite{hung1990} to study the cohomology of configuration spaces and used in~\cite{mem2009} to prove the sphere waist theorem. 

\begin{defn}
Let $Q^q(L)$ be defined inductively as follows. Take some small $\delta>0$, and let $Q^q(L)$ contain the configurations of $q$ points with following conditions:

1) for $q=1$ the space $Q^q(L)$ contains only one configuration, where one point is at the origin;

2) for $q\ge 2$ the first $q/2$ points form a configuration from $Q^{q/2}(L)$ scaled by $\delta$ and shifted by a vector $v$ of length $1$;  

3) for $q\ge 2$ the last $q/2$ points form another configuration from $Q^{q/2}(L)$ scaled by $\delta$ and shifted by $-v$.  
\end{defn}

Topologically this space is a product of $q-1$ spheres $S^{\dim L - 1}$, corresponding to different translation vectors $v$ on the stages of its construction.

If $\delta$ tends to zero, the subspace $Q^q(L)$ corresponds to binary partitions of $S^n$ by hyperplanes through the origin in $\mathbb R^{n+1}$ orthogonal to $L$ and arranged in a full binary tree of height $l$ (note $2^l=q$). The main result of~\cite{hung1990} shows that the natural map
$$
H_{\Sigma_q^{(2)}} ^*(K^q(L), \mathbb F_2) \to H_{\Sigma_q^{(2)}} ^*(Q^q(L), \mathbb F_2)
$$
is an injection, therefore Lemma~\ref{conf-sp-eu} is also valid for $Q^q(L)$, i.e. 
$$
e(\alpha_q)^{\dim L-1}\neq 0 \in H_{\Sigma_q^{(2)}}^*(Q^q(L), \mathbb F_2).
$$

Note that unlike Theorem~\ref{gromov-gen} we have to equipartition only one measure, so we may take as $L$ any linear subspace of homogeneous linear functions of dimension $m+2$. Moreover, since the space $Q^q(L)$ has hierarchical structure, we may replace $L$ by a different $(m+2)$-dimensional $L_i$ on each level $1\le i\le l-1$ of the binary tree. Denote the corresponding configuration space by $Q^q(L_1, L_2,\ldots, L_{l-1})$. This space is $\Sigma_q^{(2)}$-equivariantly homotopy equivalent to $Q^q(\mathbb R^{m+2})$, so the Euler class $e(\alpha_q)^{m+1}$ is still nonzero in its $\Sigma_q^{(2)}$-equivariant cohomology (with $\mathbb F_2$ coefficients). In~\cite{mem2009} it was shown that by selecting $L_i$ to be uniformly distributed in some sense (for large enough $l$), we obtain Claim~3 of this Lemma. Claims~1 and~2 are obtained as in the proof of Theorem~\ref{gromov-gen}; partitioning one measure ``takes'' $e(\alpha_q)$ and the coincidence in $M$ ``takes'' the remaining $e(\alpha_q)^m$.
\end{proof}

\end{document}